\pdfoutput=1
\documentclass[11pt,reqno]{amsart}

\usepackage[margin=1.0in]{geometry}
\usepackage{amsmath,amssymb,amsthm,mathtools,mathrsfs}
\usepackage{booktabs,array}
\usepackage{xcolor}
\usepackage[T1]{fontenc}
\usepackage[utf8]{inputenc}
\usepackage[colorlinks=true,citecolor=blue!60!black,linkcolor=blue!60!black,urlcolor=blue!60!black]{hyperref}
\usepackage{comment}

\numberwithin{equation}{section}

\theoremstyle{plain}
\newtheorem{theorem}{Theorem}[section]
\newtheorem{proposition}[theorem]{Proposition}
\newtheorem{lemma}[theorem]{Lemma}
\newtheorem{corollary}[theorem]{Corollary}

\theoremstyle{definition}
\newtheorem{definition}[theorem]{Definition}

\newtheorem{assumption}[theorem]{Assumption}

\theoremstyle{remark}
\newtheorem{remark}[theorem]{Remark}

\title[Quadratic one-forms on logarithmic Higgs moduli]{Quadratic one-forms on logarithmic Higgs moduli}
\author{Sumit Roy}
\address{Stat-Math Unit, Indian Statistical Institute, 203 B.T. Road, Kolkata 700 108, India.}
 \email{sumitroy\_r@isical.ac.in}
\thanks{E-mail: sumit.roy061@gmail.com}
\thanks{Address: Stat-Math Unit, Indian Statistical Institute, 203 B.T. Road, Kolkata 700 108, India.}

\begin{document}

\begin{abstract}
Let \(C\) be a compact Riemann surface of genus at least two, and let \(G\) be a connected complex reductive group. We study quadratic one-forms associated to logarithmic \(G\)-Higgs bundles on a pointed curve \((C,D)\) with nilpotent residues. We use the elementary pole cancellation for invariant polynomials, where nilpotency of the residue removes the leading pole term. In degree two this places \(B(\Phi,\Phi)\) in the cotangent space of pointed Teichmuller space, and hence gives a logarithmic quadratic one-form. We relate this one-form to the variation of the energy for tame nilpotent harmonic bundles. The energy formula is proved under a positive decay assumption near the punctures. We also write the corresponding compact formulae for classical groups and standard real forms.
\end{abstract}

\maketitle

\section{Introduction}

Let \(C\) be a compact Riemann surface of genus \(g\geq2\), and let \(G\) be a
connected complex reductive group with Lie algebra \(\mathfrak g\).  We consider
the Betti moduli space
\[
        \mathcal M_B(G)
        =
        \operatorname{Hom}(\pi_1(C),G)^{\mathrm{red}}/G
\]
and the Dolbeault moduli space \ space
\[
        \mathcal M_B(G)
        =
        \operatorname{Hom}(\pi_1(C),(\mathcal M_{\mathrm{Dol}}(G,C)
        \] of
polystable principal \(G\)-Higgs bundles.  On smooth stable loci these spaces
are identified by nonabelian Hodge theory
\cite{Hitchin1987,Donaldson,Corlette,Simpson1992,Simpson1994I,Simpson1994II}.
The Betti moduli space depends only on the topology of the surface, while the
Dolbeault moduli space depends on the complex structure of \(C\).

Hitchin studied this dependence over Teichmuller space in
\cite{HitchinUniversal}.  For \(G=GL(n,\mathbb C)\), the one-form which appears
in the first variation of the energy is
\[
        \phi(Y)
        =
        -\frac12\int_C\operatorname{tr}(\Phi^2)\mu_Y,
        \qquad
        Y\in T_C\mathcal T_g,
\]
where \(\Phi\) is the Higgs field and \(\mu_Y\) is a Beltrami representative of
\(Y\).  If \(f\) is the energy, then
\[
        Y(f)=\operatorname{Re}\phi(Y).
\]
Hitchin also averages the circle action
\((E,\Phi)\mapsto(E,e^{i\theta}\Phi)\), and obtains a connection satisfying
\[
        d_A\phi=0,\qquad
        \{\phi,\phi\}=0,\qquad
        F_A+\frac18\{\phi,\bar\phi\}=0.
\]

We first write this compact calculation for a connected complex reductive group
using a fixed invariant nondegenerate symmetric bilinear form
\[
        B:\mathfrak g\otimes\mathfrak g\longrightarrow\mathbb C.
\]
The corresponding quadratic one-form is
\[
        \phi_G(Y)
        =
        -\frac12\int_C B(\Phi,\Phi)\mu_Y.
\]
For the standard representations of
\(GL(n,\mathbb C)\), \(Sp(2n,\mathbb C)\), and \(SO(n,\mathbb C)\), we take
\(B(X,Y)=\operatorname{tr}(XY)\).  With this normalization, Hitchin's
Fourier-coefficient argument gives the same averaged-connection identities.  We
also write the corresponding formulae for the standard real forms
\(SL(2,\mathbb R)\), \(Sp(2n,\mathbb R)\), \(SO(p,q)\), and \(SU(p,q)\), using
the usual descriptions of their Higgs bundles
\cite{Hitchin1992,GPGM,ABCO,GPR}.

The main part of the paper concerns logarithmic Higgs bundles on a pointed
curve.  Let \(D=p_1+\cdots+p_r\) be a reduced divisor on \(C\).  A logarithmic
\(G\)-Higgs bundle on \((C,D)\) is a pair \((P,\Phi)\), where \(P\) is a
holomorphic principal \(G\)-bundle and
\[
        \Phi\in H^0(C,\operatorname{ad}(P)\otimes K_C(D)).
\]
Near \(p_i\), after choosing a coordinate \(z\) centred at \(p_i\), we can write
\[
        \Phi=
        \left(
        \frac{N_i}{z}+A_i(z)
        \right)dz,
        \qquad
        N_i=\operatorname{Res}_{p_i}(\Phi),
\]
with \(A_i(z)\) holomorphic.  We work with nilpotent residues.  This includes,
for example, strongly parabolic Higgs fields in the usual vector-bundle case
\cite{MehtaSeshadri,Yokogawa,BodenYokogawa}.

The local algebraic point is the following.  Let
\(p\in\mathbb C[\mathfrak g]^G\) be homogeneous of degree \(d\).  For a general
logarithmic Higgs field, \(p(\Phi)\) may have a pole of order \(d\) at a marked
point.  If the residue \(N_i\) is nilpotent, the leading coefficient vanishes.
Indeed, locally
\[
        p(\Phi)=z^{-d}p(N_i+zA_i(z))\,dz^d,
\]
and the leading term is \(p(N_i)\).  By the Jacobson--Morozov theorem, every
homogeneous invariant polynomial of positive degree vanishes on nilpotent
elements.  Thus Theorem~\ref{thm:invpoly} gives
\[
        p(\Phi)\in H^0\bigl(C,K_C^d((d-1)D)\bigr).
\]
This gives, for any connected complex reductive \(G\), the smaller pole range
for the logarithmic Hitchin map on the nilpotent-residue locus.

For the quadratic invariant \(p(X)=\frac12B(X,X)\), the estimate gives
\[
        B(\Phi,\Phi)\in H^0(C,K_C^2(D)).
\]
This is exactly the cotangent space of pointed Teichmuller space at \((C,D)\)
\[
        T^*_{(C,D)}\mathcal T_{g,r}
        \cong
        H^0(C,K_C^2(D)).
\]
Therefore \(-\frac12B(\Phi,\Phi)\) defines a logarithmic quadratic one-form
\(\phi_{G,\log}\) over pointed Teichmuller space.  If
\(Y\in H^1(C,T_C(-D))\) is represented in \v{C}ech--Dolbeault form by
\((\mu,\xi_1,\ldots,\xi_r)\), then
\[
        \phi_{G,\log}(Y)
        =
        -\frac12
        \left[
        \int_{C^\circ}B(\Phi,\Phi)\mu
        +
        \sum_{i=1}^{r}
        \operatorname{Res}_{p_i}
        \bigl(B(\Phi,\Phi)\xi_i\bigr)
        \right],
        \qquad C^\circ=C\setminus D.
\]

The pole estimate is also checked against two standard families of invariants.
For the trace invariants of \(GL(n,\mathbb C)\), the pole order \(k-1\) can
actually occur.  In other words, there are local logarithmic Higgs fields with
nilpotent residue for which \(\operatorname{tr}(\Phi^k)\) has pole order
exactly \(k-1\).  For \(SO(2m,\mathbb C)\), the Pfaffian satisfies
\[
        \operatorname{Pf}(\Phi)\in H^0(C,K_C^m((m-1)D)).
\]
If the residues also satisfy \(\operatorname{rk}N_i\leq 2m-4\) at every marked
point, then the Pfaffian loses one more pole order:
\[
        \operatorname{Pf}(\Phi)\in H^0(C,K_C^m((m-2)D)).
\]
The meromorphic Hitchin fibration over pointed curves is studied in
\cite{DonagiFernandezHerrero}.

The last part of the paper proves the logarithmic energy formula.  We use tame
nilpotent harmonic bundles on \(C^\circ=C\setminus D\) and the asymptotic
estimates of Simpson and Mochizuki
\cite{Simpson1990,MochizukiNilpotent}.  For deformations with fixed Betti
representation and fixed local monodromy conjugacy classes, we assume that the
infinitesimal complex gauge term has positive decay in the cusp scale.  This is
the only extra analytic assumption in the paper.  Under this assumption,
\[
        Y(f_{\mathrm{par}})
        =
        \operatorname{Re}\phi_{G,\log}(Y).
\]
The decay condition is used only to make the boundary terms from the
integration by parts vanish.  The pole estimate and the construction of
\(\phi_{G,\log}\) do not use it.  The same fixed-local-monodromy condition is
natural on the Betti side, where it gives the symplectic leaves of the Poisson
structure on open character varieties \cite{PantevToen}.  For real groups over
punctured curves, see \cite{BiquardGPM,GPLM}; for asymptotic geometry in the
parabolic \(SL(2,\mathbb C)\) case, see \cite{FMSW}.

The paper is organized as follows.  Section~\ref{sec:compact} contains the
compact quadratic one-form and the averaged connection.  Section
\ref{sec:classical} treats the classical groups, real forms, and the Pfaffian
Hamiltonian.  Section~\ref{sec:prelim} fixes the pointed-curve, parabolic, and
Serre-duality notation.  Section~\ref{sec:invpoly} proves the pole estimate for
invariant polynomials, gives the trace-invariant example, and proves the
Pfaffian estimate.  Section~\ref{sec:logoneform} constructs
\(\phi_{G,\log}\).  Section~\ref{sec:energy} proves the logarithmic
energy-variation formula under the cusp decay assumption.

\section{The compact quadratic one-form}
\label{sec:compact}

Let \(C\) be a compact Riemann surface of genus \(g\geq 2\) with the canonical bundle $K_C$, and let \(G\) be a connected complex reductive group with Lie algebra
\(\mathfrak g\). We fix an invariant nondegenerate symmetric bilinear form
\[
        B:\mathfrak g\otimes\mathfrak g\longrightarrow\mathbb C,
\]
which is used to define the
Atiyah-Bott-Goldman symplectic form, the energy functional, and the quadratic
Hamiltonian. For the
standard classical groups we consider $B(X,Y)=\operatorname{tr}(XY)$
in the defining representation.

\subsection{Higgs bundles and moduli spaces}
For a holomorphic principal \(G\)-bundle $P\longrightarrow C,$
the adjoint action of \(G\) on \(\mathfrak g\) gives the associated adjoint
bundle
\[
        \operatorname{ad}(P)
        =
        P\times^G\mathfrak g .
\]

\begin{definition}
A principal \(G\)-Higgs bundle on \(C\) is a pair
\(
        (P,\Phi),
\)
where \(P\) is a holomorphic principal \(G\)-bundle and
\(
        \Phi\in H^0(C,\operatorname{ad}(P)\otimes K_C)
\)
is a holomorphic section, called the Higgs field.
\end{definition}

For \(G=GL(n,\mathbb C)\), this is the usual Higgs bundle
\(
        (E,\Phi),
\) with
\(
        \Phi\in H^0(C,\operatorname{End}(E)\otimes K_C).
\) 
For \(G=SL(n,\mathbb C)\), we also have
\(
        \det E\simeq\mathcal O_C \text{ and }
        \operatorname{tr}(\Phi)=0.
\)
The topological type of a principal \(G\)-bundle over \(C\) is indexed by an
element
\(
        d\in \pi_1(G).
\)
We fix such a topological type throughout. When \(G\) is semisimple,
\(\pi_1(G)\) is finite.

Let \(Q\subset G\) be a
maximal parabolic subgroup, and let
\(
        P_Q\subset P
\)
be a holomorphic reduction of structure group to \(Q\). The reduction is said
to be preserved by the Higgs field if
\[
        \Phi\in H^0(C,\operatorname{ad}(P_Q)\otimes K_C)
\]
under the natural inclusion
\(
        \operatorname{ad}(P_Q)\subset \operatorname{ad}(P).
\)

\begin{definition}
The principal Higgs bundle \((P,\Phi)\) is called \textit{stable} (resp. \textit{semistable}), if
for every such \(\Phi\)-invariant reduction to a maximal parabolic subgroup we have
\[
    \deg(\operatorname{ad}(P_Q))<0
        \quad
        (\text{resp. } \leq 0).
\]
\end{definition}

Let
\(
        \mathcal M^d(G)
\)
be the moduli space of semistable holomorphic principal \(G\)-bundles over
\(C\) of topological type \(d\). It is a normal projective variety, and on the stable locus,
\[
        \dim_{\mathbb C}\mathcal M^d(G)
        =
        (g-1)\dim_{\mathbb C}G+\dim_{\mathbb C}Z(G),
\]
where \(Z(G)\) denotes the centre of \(G\).  We denote by
\(
        \mathcal M_{\mathrm{Dol}}^d(G,C)
\)
the smooth stable locus of the moduli space of polystable principal
\(G\)-Higgs bundles of the same topological type.

For \((P,\Phi) \in \mathcal M_{\mathrm{Dol}}^d(G,C)\), the tangent space is the first
hypercohomology of the deformation complex
\[
        \operatorname{ad}(P)
        \xrightarrow{[\Phi,\cdot]}
        \operatorname{ad}(P)\otimes K_C.
\]
Thus
\[
        T_{(P,\Phi)}\mathcal M_{\mathrm{Dol}}^d(G,C)
        \cong
        \mathbb H^1\left(
        C,\operatorname{ad}(P)
        \xrightarrow{[\Phi,\cdot]}
        \operatorname{ad}(P)\otimes K_C
        \right).
\]
On this smooth stable locus,
\[
        \dim_{\mathbb C}\mathcal M_{\mathrm{Dol}}^d(G,C)
        =
        2(g-1)\dim_{\mathbb C}G+2\dim_{\mathbb C}Z(G).
\]
In particular, if \(G\) is semisimple, then
\[
        \dim_{\mathbb C}\mathcal M_{\mathrm{Dol}}^d(G,C)
        =
        2(g-1)\dim_{\mathbb C}G.
\]

For the construction of the moduli spaces of principal bundles we refer to
Ramanathan \cite{Ramanathan, Ramanathan1996}. For Higgs bundles and the nonabelian Hodge
correspondence, see Simpson \cite{Simpson1992,Simpson1994I,Simpson1994II}. The corresponding Betti moduli space is the character variety
\[
        \mathcal M_B(G)
        =
        \operatorname{Hom}(\pi_1(C),G)^{\mathrm{red}}/G.
\]
By nonabelian Hodge theory, the smooth stable Dolbeault locus is identified
with the smooth locus of the character variety. This is obtained by solving
Hitchin's equations and passing from a harmonic metric to a flat connection
\cite{Hitchin1987,Donaldson,Corlette,Simpson1992,Simpson1994I,Simpson1994II}.
The Betti moduli space carries the Atiyah--Bott--Goldman symplectic form
defined using \(B\) \cite{AB,Goldman}. Under the nonabelian Hodge theory this is
one of the Kähler forms of the hyperkähler metric on the smooth Higgs moduli
space. For related constructions of Kähler forms on Higgs-bundle moduli spaces, see
\cite{BiswasSchumacher}.

\subsection{Hitchin's equations and the energy}

Choose a maximal compact subgroup \(K\subset G\). A \(K\)-reduction of \(P\)
defines a Chern connection \(A\), and it also defines an adjoint operation on
\(\operatorname{ad}(P)\). We denote this adjoint by \(\tau\).  For
\(G=GL(n,\mathbb C)\), if
\(
        \Phi=A(z)\,dz,
\)
then
\[
        \tau(\Phi)=A(z)^*\,d\bar z,
\]
where \(A(z)^*\) is the Hermitian adjoint.  In general,
\(
        \tau(e^{i\theta}\Phi)=e^{-i\theta}\tau(\Phi).
\) Let \(A\) be the Chern connection associated to the \(K\)-reduction, and let
\[
        F_A\in\Omega^2(C,\operatorname{ad}(P))
\]
be its curvature. Its \((0,1)\)-part is denoted by \(\bar\partial_A\). The
Hitchin equations are
\[
        F_A+[\Phi,\tau(\Phi)]=0,
        \qquad
        \bar\partial_A\Phi=0.
\]
Here the bracket combines the Lie bracket on \(\operatorname{ad}(P)\) with the
wedge product of forms. If we have
\(
        \varphi=\Phi+\tau(\Phi),
\)
then the associated flat connection is
\(
        \nabla=d_A+\varphi.
\)
Then the Hitchin equations imply
\[
        d_A\varphi=0,
        \qquad
        d_A{*}\varphi=0.
\]
We will use Hitchin's sign convention for the energy, i.e.
\begin{equation}
        f_G
        =
        -i\int_C B(\Phi,\tau(\Phi))
        =
        -\frac12\int_C B(\varphi,*\varphi).
        \label{eq:compact-energy}
\end{equation}
For \(G=GL(n,\mathbb C)\) with \(B(X,Y)=\operatorname{tr}(XY)\), this is the
normalization used in Hitchin's universal construction
\cite[Sec.~2]{HitchinUniversal}. The invariant form \(B\) defines the Atiyah-Bott-Goldman symplectic form on
the Betti moduli space.  Under nonabelian Hodge theory this is one of the
Kähler forms of the hyperkähler metric on the smooth Higgs moduli space, and we
denote it by \(\omega_1\). With respect to \(\omega_1\), the function \(f_G\) is the moment map for the
circle action
\[
        (P,\Phi)\longmapsto (P,e^{i\theta}\Phi).
\]
Indeed, the adjoint transforms by
\[
        \tau(e^{i\theta}\Phi)=e^{-i\theta}\tau(\Phi),
\]
so the circle action preserves Hitchin's equations. Let
\(
        \mathcal T=\mathcal T_g
\)
be the Teichmuller space of compact Riemann surfaces of genus \(g\). At \(C\in\mathcal T\), its tangent space is
\[
        T_C\mathcal T
        \cong
        H^1(C,T_C)
        \cong
        H^1(C,K_C^{-1}).
\]
Let
\(
        Y\in T_C\mathcal T
\)
be represented by a Beltrami differential
\(
        \mu_Y\in\Omega^{0,1}(C,T_C).
\)
The Higgs field gives a quadratic differential
\(
        B(\Phi,\Phi)\in H^0(C,K_C^2).
\)
Locally, if
\(
        \Phi=s(z)\,dz,
\)
then
\[
        B(\Phi,\Phi)=B(s(z),s(z))\,dz^2.
\]
Contracting with \(\mu_Y\) gives a \((1,1)\)-form, and we define
\begin{equation}
        \phi_G(Y)
        =
        -\frac12\int_C B(\Phi,\Phi)\,\mu_Y.
        \label{eq:compact-oneform}
\end{equation}
This is a complex-valued one-form on Teichmuller space with values in fibrewise holomorphic functions on the Higgs moduli space. For \(G=GL(n,\mathbb C)\), it is Hitchin's quadratic one-form in \cite[Sec.~2]{HitchinUniversal}.

\begin{proposition}
\label{prop:compact-variation}
Let the complex structure of \(C\) vary in the direction
\(Y\in T_C\mathcal T\), while the corresponding point of the Betti moduli space
is fixed.  Equivalently, the flat connection
\(
        \nabla=d_A+\varphi
\)
is fixed. Then
\[
        Y(f_G)
        =
        \operatorname{Re}\phi_G(Y)
        =
        -\frac12\operatorname{Re}
        \int_C B(\Phi,\Phi)\,\mu_Y.
\]
\end{proposition}

\begin{proof}
The proof follows Hitchin's first-variation calculation
\cite[Sec.~2]{HitchinUniversal}.  Since the Betti point is fixed, the
first-order change of the flat connection is represented by an infinitesimal
complex gauge transformation.  Thus one may write
\[
        \dot A=d_A\psi_1+[\varphi,\psi_2],
        \qquad
        \dot\varphi=d_A\psi_2+[\varphi,\psi_1],
\]
where \(\psi_1\) is compact-valued and \(\psi_2\) is self-adjoint.
Differentiating \eqref{eq:compact-energy} gives
\[
        \dot f_G
        =
        -\frac12\int_C B(\dot\varphi,*\varphi)
        -
        \frac12\int_C B(\varphi,\dot *\,\varphi)
        -
        \frac12\int_C B(\varphi,*\dot\varphi).
\]
The terms containing \(d_A\psi_2\) vanish by integration by parts, using
\[
        d_A\varphi=0,
        \qquad
        d_A{*}\varphi=0.
\]
The terms containing \([\varphi,\psi_1]\) vanish by the
\(\operatorname{Ad}\)-invariance of \(B\), together with the symmetry of the
pairing of one-forms defined by \(\alpha\wedge *\beta\).  Hence
\[
        \dot f_G
        =
        -\frac12\int_C B(\varphi,\dot *\,\varphi).
\]
The variation of the Hodge star in the direction represented by \(\mu_Y\) gives
\[
        \dot f_G
        =
        -\frac12\operatorname{Re}
        \int_C B(\Phi,\Phi)\,\mu_Y.
\]
Therefore
\[
        Y(f_G)=\operatorname{Re}\phi_G(Y). \qedhere
\]
\end{proof}

\subsection{The averaged connection}

The character variety \(\mathcal M_B(G)\) does not depend on the complex
structure of \(C\). Hence over Teichmuller space we have the product
fibration
\[
        \mathcal M_B(G)\times\mathcal T
        \longrightarrow
        \mathcal T.
\]
Let \(\nabla_B\) be its product flat connection. The parallel transport for
\(\nabla_B\) keeps the Betti representation fixed and changes only the complex
structure of the curve. For each \(C\in\mathcal T\), nonabelian Hodge theory identifies
\[
        \mathcal M_B(G)
        \cong
        \mathcal M_{\mathrm{Dol}}(G,C).
\]
The circle action
\(
        (P,\Phi)\longmapsto(P,e^{i\theta}\Phi)
\)
therefore induces an action on the fixed Betti fibre, and by applying this
circle action to the product connection we get a circle of flat symplectic
connections on
\[
        \mathcal M_B(G)\times\mathcal T
        \longrightarrow
        \mathcal T.
\]
Their average will be denoted by \(\nabla_A\). If
\(
        \phi_G=\beta_G+i\gamma_G,
\)
where \(\beta_G\) and \(\gamma_G\) are real Hamiltonian-valued one-forms on
\(\mathcal T\), then
\[
        \nabla_A=\nabla_B-\frac12\gamma_G .
\]
This is the averaging formula of Hitchin \cite[Sec.~3]{HitchinUniversal}. The corresponding circle of flat connections can be written as
\begin{equation}
        \nabla_\theta
        =
        \nabla_A
        -
        \frac{i}{4}
        \left(
        e^{2i\theta}\phi_G
        -
        e^{-2i\theta}\bar\phi_G
        \right),
        \qquad
        \theta\in\mathbb R/2\pi\mathbb Z.
        \label{eq:theta-connection}
\end{equation}
Here \(\phi_G\) and \(\bar\phi_G\) are Hamiltonian-valued one-forms, and hence
vertical vector-field-valued one-forms after using \(\omega_1\). For Hamiltonian functions \(h_1,h_2\) on a fibre, with
\(
        \iota_{X_{h_i}}\omega_1=dh_i,
\)
we write
\(
        \{h_1,h_2\}
        =
        \omega_1(X_{h_1},X_{h_2}).
\)
For Hamiltonian-valued differential forms on the base, this bracket is combined
with the exterior product of the base forms.

\begin{theorem}
\label{thm:compact-flatness}
On the smooth stable locus of
\(
        \mathcal M_B(G)\times\mathcal T\longrightarrow\mathcal T
\)
where the nonabelian Hodge correspondence identifies the Betti and Dolbeault
moduli spaces, the averaged connection satisfies
\[
        d_A\phi_G=0,
        \qquad
        \{\phi_G,\phi_G\}=0,
        \qquad
        F_A+\frac18\{\phi_G,\bar\phi_G\}=0.
\]
\end{theorem}

\begin{proof}
For \(G=GL(n,\mathbb C)\) and \(B(X,Y)=\operatorname{tr}(XY)\), these are the
identities obtained by Hitchin from the circle family of flat connections
\cite[Sec.~3]{HitchinUniversal}. The same computation applies here. Indeed, the circle action gives
\[
        B(e^{i\theta}\Phi,e^{i\theta}\Phi)
        =
        e^{2i\theta}B(\Phi,\Phi),
\]
and the form \(B\) is the same one used in the symplectic form, in the energy,
and in the Hamiltonian function \(\phi_G\). Therefore the Hamiltonian one-form
has circle weight two. Let
\[
        a_\theta
        =
        -
        \frac{i}{4}
        \left(
        e^{2i\theta}\phi_G
        -
        e^{-2i\theta}\bar\phi_G
        \right),
        \qquad
        \nabla_\theta=\nabla_A+a_\theta .
\]
The connections \(\nabla_\theta\) are flat for all \(\theta\), since they are
obtained from the product flat connection by the circle action. Hence
\[
        0
        =
        F_{\nabla_\theta}
        =
        F_A+d_Aa_\theta+\frac12\{a_\theta,a_\theta\}.
\]
This identity is a finite Fourier series in \(e^{2i\theta}\).  Comparing its
Fourier coefficients gives the three equations. The coefficient of
\(e^{2i\theta}\) gives
\[
        d_A\phi_G=0,
\]
the coefficient of \(e^{4i\theta}\) gives
\[
        \{\phi_G,\phi_G\}=0,
\]
and the constant coefficient gives
\[
        F_A+\frac18\{\phi_G,\bar\phi_G\}=0.
\]
\end{proof}

\section{Classical groups, real forms, and the Pfaffian}
\label{sec:classical}

We now consider the classical groups and some real forms.  For the standard
representations of
\(
        GL(n,\mathbb C), \ Sp(2n,\mathbb C),\text{ and } SO(n,\mathbb C),
\)
we use
\(
        B(X,Y)=\operatorname{tr}(XY).
\)
Thus the quadratic Hamiltonian is given by \(\operatorname{tr}(\Phi^2)\). For a general invariant bilinear form \(B\), the same formulae hold with
\(\operatorname{tr}(\Phi^2)\) replaced by \(B(\Phi,\Phi)\). We shall also use the Hitchin morphism. Let
\(
        p_1,\ldots,p_\ell
\)
be homogeneous generators of the invariant ring
\(
        \mathbb C[\mathfrak g]^G
\)
with degrees
\(
        d_1,\ldots,d_\ell.
\)
The Hitchin morphism is
\[
        h_G:\mathcal M_{\mathrm{Dol}}(G,C)
        \longrightarrow
        \mathcal A_G
        :=
        \bigoplus_{j=1}^{\ell}H^0(C,K_C^{d_j}),
\]
defined by
\[
        h_G(P,\Phi)
        =
        \bigl(p_1(\Phi),\ldots,p_\ell(\Phi)\bigr).
\]
The vector space \(\mathcal A_G\) is called the Hitchin base (see
\cite{Hitchin1987,DonagiPantev} for more details).

\subsection{Complex symplectic and orthogonal groups}

\subsubsection{\(G=Sp(2n,\mathbb C)\)} An \(Sp(2n,\mathbb C)\)-Higgs bundle is a triple
\(
        (E,\omega,\Phi),
\)
where \(E\) is a holomorphic vector bundle of rank \(2n\),
\[
        \omega:E\otimes E\longrightarrow \mathcal O_C
\]
is a nondegenerate skew-symmetric form, and
\(
        \Phi\in H^0(C,\operatorname{End}(E)\otimes K_C)
\)
satisfies
\[
        \omega(\Phi u,v)+\omega(u,\Phi v)=0
\]
for local sections \(u,v\) of \(E\). Equivalently,
\(
        \Phi\in H^0(C,\mathfrak{sp}(E,\omega)\otimes K_C).
\)
The quadratic one-form is
\[
        \phi_{Sp}(Y)
        =
        -\frac12\int_C\operatorname{tr}(\Phi^2)\mu_Y,
\]
and the Hitchin base is
\[
        \mathcal A_{Sp(2n)}
        =
        \bigoplus_{j=1}^{n}H^0(C,K_C^{2j}).
\]

\subsubsection{\(G=SO(n,\mathbb C)\), \(n\geq3\)} An \(SO(n,\mathbb C)\)-Higgs bundle is a triple
\(
        (E,q,\Phi),
\)
where \(E\) is a holomorphic vector bundle of rank \(n\),
\(
        q:E\otimes E\longrightarrow\mathcal O_C
\)
is a nondegenerate symmetric form with fixed orientation, and
\(
        \Phi\in H^0(C,\operatorname{End}(E)\otimes K_C)
\)
satisfies
\(
        q(\Phi u,v)+q(u,\Phi v)=0
\)
for local sections \(u,v\) of \(E\). Equivalently,
\(
        \Phi\in H^0(C,\mathfrak{so}(E,q)\otimes K_C).
\)
The quadratic one-form is
\[
        \phi_{SO}(Y)
        =
        -\frac12\int_C\operatorname{tr}(\Phi^2)\mu_Y.
\]
For \(SO(2m+1,\mathbb C)\),
\[
        \mathcal A_{SO(2m+1)}
        =
        \bigoplus_{j=1}^{m}H^0(C,K_C^{2j}),
\]
while for \(SO(2m,\mathbb C)\),
\[
        \mathcal A_{SO(2m)}
        =
        \bigoplus_{j=1}^{m-1}H^0(C,K_C^{2j})
        \oplus H^0(C,K_C^m),
\]
the last summand being the Pfaffian invariant.

In both symplectic and orthogonal cases the circle action preserves the defining symplectic or orthogonal
condition on \(\Phi\). Hence the averaged connection and the identities of
Theorem~\ref{thm:compact-flatness} restrict to every smooth stable
\(Sp(2n,\mathbb C)\) or \(SO(n,\mathbb C)\)-component. For the classical
Hitchin bases and invariant degrees, see \cite{Hitchin1987}.

\subsection{Real forms}

Let \(G_{\mathbb R}\) be a real reductive group, let
\(
        H\subset G_{\mathbb R}
\)
be a maximal compact subgroup, and let
\[
        \mathfrak g_{\mathbb R}
        =
        \mathfrak h\oplus\mathfrak m
\]
be the Cartan decomposition.  
\begin{definition}
A \(G_{\mathbb R}\)-Higgs bundle is a pair
\(
        (P_H,\Phi),
\)
where \(P_H\) is a holomorphic principal \(H^{\mathbb C}\)-bundle and
\(
        \Phi\in H^0(C,P_H(\mathfrak m^{\mathbb C})\otimes K_C),
\)
where
\[
        P_H(\mathfrak m^{\mathbb C})
        =
        P_H\times^{H^{\mathbb C}}\mathfrak m^{\mathbb C}
\]
is the associated vector bundle.
\end{definition}

The circle action preserves the real-form Higgs locus, since
\[
        e^{i\theta}\Phi
        \in H^0(C,P_H(\mathfrak m^{\mathbb C})\otimes K_C)
\]
whenever
\[
        \Phi\in H^0(C,P_H(\mathfrak m^{\mathbb C})\otimes K_C).
\]
Therefore the circle-translated product connections are tangent to every smooth
real-form component, and so is their average. Hence the connection
\[
        \nabla_A=\nabla_B-\frac12\gamma
\]
and the identities of Theorem~\ref{thm:compact-flatness} restrict to these
components. We now write the quadratic one-form in the standard block
descriptions of the real-form Higgs fields.

\subsubsection{\texorpdfstring{\(SL(2,\mathbb R)\)}{SL(2,R)}}

For the uniformizing component of the \(SL(2,\mathbb R)\)-Higgs moduli space,
choose a square root \(K_C^{1/2}\). The corresponding Higgs bundle is
\[
        E=K_C^{1/2}\oplus K_C^{-1/2},
        \qquad
        \Phi=
        \begin{pmatrix}
        0&q\\
        1&0
        \end{pmatrix},
        \qquad
        q\in H^0(C,K_C^2).
\]
Here the lower-left entry is the natural section of
\[
        \operatorname{Hom}(K_C^{1/2},K_C^{-1/2})\otimes K_C
        \cong
        \mathcal O_C.
\]
A direct calculation gives
\[
        \Phi^2=
        \begin{pmatrix}
        q&0\\
        0&q
        \end{pmatrix},
        \qquad
        \operatorname{tr}(\Phi^2)=2q.
\]
Therefore
\[
        \phi_{SL(2,\mathbb R)}(Y)
        =
        -\int_C q\,\mu_Y.
\]
Under
\(
        T_C^*\mathcal T\cong H^0(C,K_C^2),
\)
this is the tautological one-form on \(T^*\mathcal T\), up to the sign used in
our definition of \(\phi_G\).

\subsubsection{\texorpdfstring{\(Sp(2n,\mathbb R)\)}{Sp(2n,R)}}

An \(Sp(2n,\mathbb R)\)-Higgs bundle is given by
\(
        (V,\beta,\gamma),
\)
where \(V\) is a holomorphic vector bundle of rank \(n\),
\[
        \beta\in H^0(C,\operatorname{Sym}^2V\otimes K_C),
        \qquad
        \gamma\in H^0(C,\operatorname{Sym}^2V^*\otimes K_C).
\]
The Higgs field has the block form
\[
        \Phi=
        \begin{pmatrix}
        0&\beta\\
        \gamma&0
        \end{pmatrix}.
\]
Thus
\[
        \Phi^2=
        \begin{pmatrix}
        \beta\gamma&0\\
        0&\gamma\beta
        \end{pmatrix}.
\]
Therefore, we get
\[
        \operatorname{tr}(\Phi^2)
        =
        2\operatorname{tr}(\beta\gamma).
\]
Hence
\[
        \phi_{Sp(2n,\mathbb R)}(Y)
        =
        -\int_C\operatorname{tr}(\beta\gamma)\mu_Y.
\]
For \(n=1\), this agrees with the \(SL(2,\mathbb R)\) formula.  For
\(n\geq2\), the Hitchin base contains higher even differentials, but the
variation over Teichmuller space uses only the quadratic invariant. For the
moduli spaces of \(Sp(2n,\mathbb R)\)-Higgs bundles, see
García-Prada--Gothen--Mundet \cite{GPGM}.

\subsubsection{\texorpdfstring{\(SO(p,q)\)}{SO(p,q)}}

An \(SO(p,q)\)-Higgs bundle is a tuple
\(
        (V,W,q_V,q_W,\eta),
\)
where \(V\) and \(W\) are holomorphic orthogonal bundles of ranks \(p\) and
\(q\), respectively,
\[
        \det(V)\otimes\det(W)\cong\mathcal O_C,
\]
and
\[
        \eta\in H^0(C,\operatorname{Hom}(W,V)\otimes K_C).
\]
Using the orthogonal forms, \(\eta\) has a transpose
\[
        \eta^t:V\longrightarrow W\otimes K_C.
\]
Therefore, the Higgs field is given by
\[
        \Phi=
        \begin{pmatrix}
        0&\eta\\
        -\eta^t&0
        \end{pmatrix}.
\]
Then
\[
        \Phi^2=
        \begin{pmatrix}
        -\eta\eta^t&0\\
        0&-\eta^t\eta
        \end{pmatrix},
\]
and therefore
\[
        \operatorname{tr}(\Phi^2)
        =
        -2\operatorname{tr}(\eta\eta^t).
\]
Thus, for the trace form,
\[
        \phi_{SO(p,q)}(Y)
        =
        \int_C\operatorname{tr}(\eta\eta^t)\mu_Y.
\]
Equivalently, without choosing signs in the block notation, one may write the
same one-form invariantly as
\[
        \phi_{SO(p,q)}(Y)
        =
        -\frac12\int_C B(\Phi,\Phi)\mu_Y.
\]
For \(SO(p,q)\)-Higgs bundles and higher Teichmuller components, see
\cite{ABCO}.

\subsubsection{\texorpdfstring{\(SU(p,q)\)}{SU(p,q)}}

An \(SU(p,q)\)-Higgs bundle consists of holomorphic vector bundles \(V\) and
\(W\), of ranks \(p\) and \(q\), with
\[
        \det(V)\otimes\det(W)\cong\mathcal O_C,
\]
and Higgs field
\[
        \Phi=
        \begin{pmatrix}
        0&\beta\\
        \gamma&0
        \end{pmatrix},
\]
where
\[
        \beta\in H^0(C,\operatorname{Hom}(W,V)\otimes K_C),
        \qquad
        \gamma\in H^0(C,\operatorname{Hom}(V,W)\otimes K_C).
\]
Then
\[
        \Phi^2=
        \begin{pmatrix}
        \beta\gamma&0\\
        0&\gamma\beta
        \end{pmatrix},
        \qquad
        \operatorname{tr}(\Phi^2)
        =
        2\operatorname{tr}(\beta\gamma).
\]
Hence
\[
        \phi_{SU(p,q)}(Y)
        =
        -\int_C\operatorname{tr}(\beta\gamma)\mu_Y.
\]

\subsection{The Pfaffian Hamiltonian}

Let
\(
        G=SO(2m,\mathbb C),\ m\geq 2.
\)
The Hitchin base contains the Pfaffian summand
\(
        H^0(C,K_C^m),
\)
corresponding to the degree \(m\) invariant
\[
        \operatorname{Pf}(\Phi)\in H^0(C,K_C^m).
\]
For \(m>2\), this invariant is independent of the quadratic invariant.  For
\(m=2\), the exceptional isomorphism
\[
        \mathfrak{so}(4,\mathbb C)\cong
        \mathfrak{sl}_2(\mathbb C)\oplus\mathfrak{sl}_2(\mathbb C)
\]
places the Pfaffian in degree two.

We use Hitchin's hyperk\"ahler notation.  The complex structure \(I\) is the
Dolbeault complex structure, \(J\) is the de Rham complex structure, and
\(K=IJ\).  The corresponding K\"ahler forms are denoted by
\[
        \omega_1,\qquad \omega_2,\qquad \omega_3,
\]
and
\[
        \omega^c=\omega_2+i\omega_3
\]
is the holomorphic symplectic form for the Dolbeault complex structure.

Let \(h\) be a fibrewise holomorphic function on the Higgs moduli space, and
let \(Z_h\) be its Hamiltonian vector field with respect to
\(
        \omega^c
        \ \text{and}\
        \iota_{Z_h}\omega^c=dh.
\)
The hyperk\"ahler identity
\[
        Jdh=2\,\iota_{Z_h}\omega_1
\]
gives a closed two-form
\(
        dJdh
\)
(see \cite[Sec.~2.3]{HitchinUniversal}). It is of type \((1,1)\) for the complex structures in the hyperk\"ahler family. 

Let
\(
        \nu\in H^1(C,K_C^{1-m}).
\)
Since
\(
        \operatorname{Pf}(\Phi)\in H^0(C,K_C^m),
\)
Serre duality gives a function on the smooth \(SO(2m,\mathbb C)\)-Higgs moduli
space
\[
        h_{\operatorname{Pf},\nu}
        =
        \left\langle
        \operatorname{Pf}(\Phi),\nu
        \right\rangle .
\]
The brackets denote the Serre-duality pairing
\[
        H^0(C,K_C^m)\times H^1(C,K_C^{1-m})
        \longrightarrow \mathbb C .
\]

\begin{proposition}
\label{prop:pfaffian-compact}
The function \(h_{\operatorname{Pf},\nu}\) is fibrewise holomorphic.  The
associated two-form
\(
        dJdh_{\operatorname{Pf},\nu}
\)
is closed and of type \((1,1)\) on the smooth \(SO(2m,\mathbb C)\)-locus.
\end{proposition}
\begin{proof}
The Pfaffian is a polynomial invariant of \(\mathfrak{so}(2m,\mathbb C)\). Hence \(\operatorname{Pf}(\Phi)\), and therefore
\(h_{\operatorname{Pf},\nu}\), depends holomorphically on the Higgs field along
each Dolbeault fibre. Let \(Z_{h_{\operatorname{Pf},\nu}}\) be the Hamiltonian vector field with
respect to
\[
        \omega^c=\omega_2+i\omega_3,
        \qquad
        \iota_{Z_{h_{\operatorname{Pf},\nu}}}\omega^c
        =
        dh_{\operatorname{Pf},\nu}.
\]
The hyperk\"ahler identity
\[
        Jdh_{\operatorname{Pf},\nu}
        =
        2\,\iota_{Z_{h_{\operatorname{Pf},\nu}}}\omega_1
\]
implies that
\(
        dJdh_{\operatorname{Pf},\nu}
\)
is closed and of type \((1,1)\).  If \(h_{\operatorname{Pf},\nu}\) is
complex-valued, the statement is understood by applying the same argument to
its real and imaginary parts.
\end{proof}

For \(m>2\), the Pfaffian Hamiltonian is different from ordinary Teichmuller
variation.  Its parameter lies in
\(
        H^1(C,K_C^{1-m}),
\)
whereas
\(
        T_C\mathcal T\cong H^1(C,K_C^{-1}).
\)
Thus the Pfaffian gives a higher Hitchin Hamiltonian.  When \(m=2\), the
exceptional isomorphism
\[
        \mathfrak{so}(4,\mathbb C)
        \cong
        \mathfrak{sl}_2(\mathbb C)\oplus\mathfrak{sl}_2(\mathbb C)
\]
puts the Pfaffian in degree two, and the two parameter spaces coincide.

\section{Pointed curves and parabolic Higgs bundles}
\label{sec:prelim}

Let
\(
        D=p_1+\cdots+p_r
\)
be a reduced divisor on \(C\), and put
\(
        C^\circ=C\setminus D.
\)
A logarithmic Higgs field on \((C,D)\) is allowed simple poles along \(D\). In a
coordinate \(z\) centred at \(p_i\), it has the form
\[
        \Phi=
        \left(
        \frac{N_i}{z}+A_i(z)
        \right)dz,
        \qquad
        N_i=\operatorname{Res}_{p_i}(\Phi).
\]
The nilpotency of \(N_i\) controls the leading pole terms of invariant
polynomials in \(\Phi\).

\subsection{Parabolic Higgs bundles}

\begin{definition}
A \textit{parabolic bundle} over \((C,D)\) is a holomorphic vector bundle \(E\)
on \(C\), together with, for every \(p_i\in D\), a decreasing filtration
\[
        E_{p_i}=F_{i,1}\supset F_{i,2}\supset\cdots
        \supset F_{i,\ell_i}\supset F_{i,\ell_i+1}=0
\]
and real numbers
\[
        0\leq \alpha_{i,1}<\alpha_{i,2}<\cdots<\alpha_{i,\ell_i}<1.
\]
The numbers \(\alpha_{i,j}\) are called the \textit{parabolic weights}.
\end{definition}

The multiplicity of the weight \(\alpha_{i,j}\) is
\[
        m_{i,j}
        =
        \dim(F_{i,j}/F_{i,j+1}).
\]
The parabolic degree of \(E\) is
\[
        \operatorname{pardeg}(E)
        =
        \deg(E)
        +
        \sum_{i=1}^{r}\sum_{j=1}^{\ell_i}
        \alpha_{i,j}m_{i,j},
\]
and the parabolic slope is
\[
        \operatorname{par}\mu(E)
        =
        \frac{\operatorname{pardeg}(E)}{\operatorname{rk}(E)}.
\]

If \(F\subset E\) is a holomorphic subbundle, then \(F\) has an induced
parabolic structure.  At \(p_i\), the filtration is obtained by intersecting
with the filtration of \(E_{p_i}\)
\[
        F_{p_i}\cap F_{i,1}\supset
        F_{p_i}\cap F_{i,2}\supset\cdots\supset
        F_{p_i}\cap F_{i,\ell_i}\supset0,
\]
after deleting repetitions. The weights are the corresponding weights of
\(E\). The parabolic degree and parabolic slope of \(F\) are computed with
this induced parabolic structure.

\begin{definition}
A logarithmic Higgs field on \(E\) is a section
\(
        \Phi\in H^0(C,\operatorname{End}(E)\otimes K_C(D)).
\)
A \textit{parabolic Higgs bundle} is a parabolic vector bundle \(E\) together with a
logarithmic Higgs field \(\Phi\) such that, for every \(p_i\in D\),
\[
        \operatorname{Res}_{p_i}(\Phi)(F_{i,j})
        \subset F_{i,j}
        \qquad
        \text{for all }j.
\]
It is called \textit{strongly parabolic} if
\[
        \operatorname{Res}_{p_i}(\Phi)(F_{i,j})
        \subset F_{i,j+1}
        \qquad
        \text{for all }i,j.
\]
\end{definition}

Thus, for a strongly parabolic Higgs field, the residue at \(p_i\) is nilpotent
with respect to the flag at \(p_i\). In particular, if the flag is full, then
the residue is represented by a strictly upper triangular matrix in a basis
adapted to the flag.

\begin{definition}
A parabolic subbundle \(F\subset E\) is called \(\Phi\)-invariant if
\(
        \Phi(F)\subset F\otimes K_C(D).
\)
A parabolic Higgs bundle \((E,\Phi)\) is called \textit{stable} (resp. \textit{semistable}), if for
every proper nonzero \(\Phi\)-invariant parabolic subbundle \(F\subset E\), we have
\[
        \operatorname{par}\mu(F)<\operatorname{par}\mu(E),
        \quad (\text{resp.} \leq)        
\]
\end{definition}

For fixed rank, degree, parabolic weights and multiplicities, the moduli space
of semistable parabolic Higgs bundles is a quasi-projective variety.  We shall
only use its smooth stable locus (see \cite{MehtaSeshadri, Yokogawa, BodenYokogawa, BiswasParabolic} for more details). For the relation between parabolic Higgs bundles and Teichmuller
spaces of punctured surfaces, see \cite{BiswasGastesiGovindarajan}.

\begin{remark}
For \(SL(2,\mathbb C)\), we have
\(
        \det E\simeq\mathcal O_C,
        \ \text{and} \
        \Phi\in H^0(C,\operatorname{End}_0(E)\otimes K_C(D)).
\)
If \((E,\Phi)\) is strongly parabolic and
\(
        N_i=\operatorname{Res}_{p_i}(\Phi),
\)
then
\(
        N_i\in\mathfrak{sl}_2(\mathbb C).
\)
Since \(N_i\) strictly lowers the parabolic flag, it is nilpotent, in rank
two this gives
\(
        N_i^2=0,
\)
or equivalently
\(
        \operatorname{tr}(N_i^2)=0.
\)
This observation will be used later for
\(\operatorname{tr}(\Phi^2)\).
\end{remark}

\subsection{Logarithmic principal Higgs bundles}

\begin{definition}
A \textit{logarithmic \(G\)-Higgs bundle} on \((C,D)\) is a pair
\(
        (P,\Phi),
\)
where \(P\) is a holomorphic principal \(G\)-bundle on \(C\), and
\(
        \Phi\in H^0(C,\operatorname{ad}(P)\otimes K_C(D)).
\)
\end{definition}

Let \(p_i\in D\). Choose a coordinate \(z\) centred at \(p_i\), and choose a
local trivialization of \(\operatorname{ad}(P)\). Then the Higgs field has the
local form
\[
        \Phi=
        \left(
        \frac{N_i}{z}+A_i(z)
        \right)dz,
\]
where \(A_i(z)\) is holomorphic. The element
\(
        N_i\in\mathfrak g
\)
is the residue of \(\Phi\) at \(p_i\), and we write
\(
        N_i=\operatorname{Res}_{p_i}(\Phi).
\)

\begin{definition}
We say that the logarithmic Higgs field \(\Phi\) has nilpotent residues if
\(
        \operatorname{Res}_{p_i}(\Phi)
\)
is nilpotent for every \(p_i\in D\).
\end{definition}

For a reductive Lie algebra
\(
        \mathfrak g
        =
        \mathfrak z(\mathfrak g)\oplus[\mathfrak g,\mathfrak g],
\)
where \(\mathfrak z(\mathfrak g)\) is the centre and
\([\mathfrak g,\mathfrak g]\) is semisimple, an element
\(
        N\in\mathfrak g
\)
is called nilpotent if its component in \(\mathfrak z(\mathfrak g)\) is zero
and its component in \([\mathfrak g,\mathfrak g]\) is nilpotent.

We fix a smooth stable family
\[
        \mathcal M^{\mathrm{nil}}_{\mathrm{Dol}}(G)
        \longrightarrow
        \mathcal T_{g,r},
\]
whose fibre over \((C,D)\) is a smooth stable locus in the logarithmic Dolbeault
moduli space of \(G\)-Higgs bundles with fixed topological type and nilpotent
residues.

\subsection{Pointed Teichmuller space}

Let \(\mathcal T_{g,r}\) denote the Teichmuller space of pointed Riemann
surfaces of type \((g,r)\).  At the point represented by \((C,D)\), one has
\[
        T_{(C,D)}\mathcal T_{g,r}
        \cong
        H^1(C,T_C(-D)).
\]
Since
\[
        (T_C(-D))^\vee=K_C(D),
\]
Serre duality gives
\[
        T^*_{(C,D)}\mathcal T_{g,r}
        \cong
        H^0(C,K_C^2(D)).
\]
Thus the cotangent space is the space of meromorphic quadratic differentials
with at most simple poles at the marked points.
This is the standard cotangent-space description of pointed Teichmuller space;
see, for example, \cite{Gardiner,Hubbard}.

We shall use two representatives of a tangent vector
\(
        Y\in H^1(C,T_C(-D)).
\)
The first is a Dolbeault representative
\(
        \nu\in A^{0,1}(C,T_C(-D)),
\)
and the second is a \v{C}ech--Dolbeault representative
\(
        (\mu,\xi_1,\ldots,\xi_r),
\)
where \(\mu\) is a Beltrami differential on \(C^\circ\), and \(\xi_i\) is a
local holomorphic vector field near \(p_i\).  The vector
\(
        \xi_i(p_i)\in T_{p_i}C
\)
is the tangent direction of the marked point \(p_i\).

The following lemma will be used in the proof of the logarithmic energy-variation
formula.

\begin{lemma}
\label{lem:vanishing-representative}
Every class
\(
        Y\in H^1(C,T_C(-D))
\)
admits a Dolbeault representative
\(
        \nu\in A^{0,1}(C,T_C(-D))
\)
which is identically zero on a neighbourhood of \(D\).
\end{lemma}

\begin{proof}
Let
\(
        \nu_0\in A^{0,1}(C,T_C(-D))
\)
be a \(\bar\partial\)-closed representative of \(Y\). Choose pairwise disjoint
coordinate discs \(U_i\) around the points \(p_i\). Since \(U_i\) is
contractible, the local \(\bar\partial\)-Poincare lemma for the line bundle
\(T_C(-D)\) gives
\(
        \eta_i\in A^{0,0}(U_i,T_C(-D))
\)
such that
\(
        \bar\partial\eta_i=\nu_0|_{U_i}.
\)
Choose a smooth function \(\chi_i\) on \(C\), supported in \(U_i\), such that
\[
        \chi_i=1
        \quad\text{on a smaller disc }U_i'\subset U_i.
\] Define
\[
        \nu
        =
        \nu_0-\sum_{i=1}^{r}\bar\partial(\chi_i\eta_i).
\]
Then \(\nu\) differs from \(\nu_0\) by a \(\bar\partial\)-exact term, and hence
represents the same class in \(H^1(C,T_C(-D))\).  On \(U_i'\), we have
\(\chi_i=1\), and therefore
\[
        \nu
        =
        \nu_0-\bar\partial\eta_i
        =
        0.
\]
Thus \(\nu\) vanishes on a neighbourhood of \(D\).
\end{proof}

\subsection{Residue and the local Serre pairing}

For a meromorphic one-form \(\alpha\) near \(p\), we write
\(
        \operatorname{Res}_{p}(\alpha)
\)
for the coefficient of \(dz/z\) in a local coordinate \(z\) centred at \(p\). Equivalently,
\[
        \operatorname{Res}_{p}(\alpha)
        =
        \frac{1}{2\pi i}
        \int_{|z|=\varepsilon}\alpha .
\]
Let
\(
        q\in H^0(C,K_C^2(D)),
\)
and let
\(
        Y\in H^1(C,T_C(-D))
\)
be represented by
\(
        (\mu,\xi_1,\ldots,\xi_r),
\)
where \(\mu\) is a Beltrami differential on \(C^\circ\), and each \(\xi_i\) is
a local holomorphic vector field near \(p_i\). Then the Serre pairing is
\begin{equation}
        \langle q,Y\rangle
        =
        \int_{C^\circ}q\,\mu
        +
        \sum_{i=1}^{r}
        \operatorname{Res}_{p_i}(q\,\xi_i).
        \label{eq:serre-local}
\end{equation}
This is the \v{C}ech--Dolbeault form of the pairing
\[
        H^0(C,K_C^2(D))\times H^1(C,T_C(-D))
        \longrightarrow
        \mathbb C.
\]
This is the usual \v{C}ech--Dolbeault expression for the Serre pairing; see,
for example, \cite[Ch.~2]{GriffithsHarris}.
The residue term depends only on the tangent vector
\(
        \xi_i(p_i)\in T_{p_i}C.
\)
Indeed, if \(\eta_i\) is a local holomorphic vector field with
\(
        \eta_i(p_i)=0,
\)
then \(q\,\eta_i\) is holomorphic at \(p_i\), since \(q\) has at most a simple
pole. Therefore
\[
        \operatorname{Res}_{p_i}(q\,\eta_i)=0.
\]

\section{Nilpotent residues and invariant polynomials}
\label{sec:invpoly}

We use the notion of nilpotent residue fixed in Section~\ref{sec:prelim}. The
following lemma is about the quadratic invariant associated to an invariant
symmetric bilinear form.

We shall use the following consequence of the Jacobson-Morozov
theorem: if \(N\) is a nilpotent element of a complex reductive Lie algebra in
the sense fixed above, then there exists
\(
        H\in[\mathfrak g,\mathfrak g]
\)
such that
\(
        [H,N]=2N.
\)
(see \cite[Ch.~3]{CollingwoodMcGovern}).

\begin{proposition}
\label{prop:nilpotent-isotropic}
Let
\(
        B:\mathfrak g\otimes\mathfrak g\longrightarrow\mathbb C
\)
be an invariant symmetric bilinear form. If \(N\in\mathfrak g\) is nilpotent,
then
\(
        B(N,N)=0.
\)
\end{proposition}

\begin{proof}
By the definition fixed in Section~\ref{sec:prelim}, the central component of
\(N\) is zero. Hence
\(
        N\in[\mathfrak g,\mathfrak g].
\)
Choose
\(
        H\in[\mathfrak g,\mathfrak g]
\)
such that
\(
        [H,N]=2N.
\)
Using the invariance of \(B\), we get
\[
        2B(N,N)
        =
        B([H,N],N)
        =
        B(H,[N,N])
        =
        0.
\]
Therefore
\(
        B(N,N)=0. \qedhere
\)
\end{proof}

We also need the corresponding vanishing for homogeneous invariant polynomials.

\begin{theorem}
\label{thm:invpoly}
Let
\(
        p\in\mathbb C[\mathfrak g]^G
\)
be a homogeneous invariant polynomial of degree \(d\geq 1\). Let
\((P,\Phi)\) be a logarithmic \(G\)-Higgs bundle on \((C,D)\). Assume that
every residue
\(
        N_i=\operatorname{Res}_{p_i}(\Phi)
\)
is nilpotent. Then
\[
        p(\Phi)\in H^0\bigl(C,K_C^d((d-1)D)\bigr).
\]
\end{theorem}

\begin{proof}
Fix \(p_i\in D\).  Choose a coordinate \(z\) centred at \(p_i\), and choose a
local trivialization of \(\operatorname{ad}(P)\).  Then
\[
        \Phi=
        \left(
        \frac{N_i}{z}+A_i(z)
        \right)dz,
\]
where \(A_i(z)\) is holomorphic and
\(
        N_i=\operatorname{Res}_{p_i}(\Phi).
\)
Since \(p\) is homogeneous of degree \(d\),
\[
        p(\Phi)
        =
        z^{-d}p(N_i+zA_i(z))\,dz^d.
\]
We first show that
\(
        p(N_i)=0.
\)
By the nilpotency of \(N_i\) and the Jacobson-Morozov consequence, there exists
\(
        H_i\in[\mathfrak g,\mathfrak g]
\)
such that
\(
        [H_i,N_i]=2N_i.
\)
Since \(p\) is invariant under the adjoint action,
\[
        p(N_i)
        =
        p(e^{s\operatorname{ad}H_i}N_i)
        =
        p(e^{2s}N_i)
        =
        e^{2ds}p(N_i)
\]
for every \(s\in\mathbb C\).  Since \(d>0\), this forces
\(
        p(N_i)=0,
\)
which is the constant term of
\(
        p(N_i+zA_i(z)).
\)
Hence
\[
        p(N_i+zA_i(z))=O(z),
\]
and therefore
\[
        p(\Phi)=O(z^{-(d-1)})\,dz^d.
\]
Thus \(p(\Phi)\) has pole order at most \(d-1\) at \(p_i\).  Since this holds
for every \(p_i\in D\), we obtain
\[
        p(\Phi)\in H^0\bigl(C,K_C^d((d-1)D)\bigr).
\]
\end{proof}

We shall also use the following local refinement of the same argument.

\begin{lemma}
\label{lem:refined-pole}
Let \(p\in\mathbb C[\mathfrak g]^G\) be homogeneous of degree \(d\). Fix
\(p_i\in D\). Suppose that \(1\leq r_i\leq d\) and 
\[
        (d^j p)_{N_i}=0
        \qquad
        \text{for }0\leq j\leq r_i-1.
\]
Then \(p(\Phi)\) has pole order at most \(d-r_i\) at \(p_i\).
\end{lemma}

\begin{proof}
With the notation of the above proof,
\[
        p(\Phi)
        =
        z^{-d}p(N_i+zA_i(z))\,dz^d.
\]
Then the vanishing of the first \(r_i\) Taylor terms gives
\[
        p(N_i+zA_i(z))=O(z^{r_i}).
\]
Hence
\[
        p(\Phi)=O(z^{-(d-r_i)})\,dz^d.
\]
\end{proof}

\begin{corollary}
\label{cor:nilpotent-hitchin-map}
Let
\(
        p_1,\ldots,p_\ell
\)
be homogeneous generators of \(\mathbb C[\mathfrak g]^G\), with degrees
\(
        d_1,\ldots,d_\ell.
\)
On the nilpotent-residue locus, the logarithmic Hitchin map takes values in
\[
        \bigoplus_{j=1}^{\ell}
        H^0\bigl(C,K_C^{d_j}((d_j-1)D)\bigr).
\]
Equivalently, the usual logarithmic Hitchin map factors as
\[
        h_{G,\log}^{\mathrm{nil}}:
        \mathcal M_{\mathrm{Dol}}^{\mathrm{nil}}(G)
        \longrightarrow
        \bigoplus_{j=1}^{\ell}
        H^0\bigl(C,K_C^{d_j}((d_j-1)D)\bigr).
\]
\end{corollary}

\begin{proof}
This follows by applying Theorem~\ref{thm:invpoly} to each generator
\(p_j\).
\end{proof}

\begin{corollary}
\label{cor:quadratic}
Let
\(
        B:\mathfrak g\otimes\mathfrak g\longrightarrow\mathbb C
\)
be an invariant symmetric bilinear form.  Let \((P,\Phi)\) be a logarithmic
\(G\)-Higgs bundle on \((C,D)\) with nilpotent residues. Then
\[
        B(\Phi,\Phi)\in H^0(C,K_C^2(D)).
\]
\end{corollary}

\begin{proof}
The polynomial
\[
        p(X)=\frac12 B(X,X)
\]
is invariant and homogeneous of degree two.  The result follows from
Theorem~\ref{thm:invpoly}.
\end{proof}

\subsection{Sharpness for trace invariants}

For \(G=GL(n,\mathbb C)\), Theorem~\ref{thm:invpoly} applied to
\(
        p(X)=\operatorname{tr}(X^k)
\)
gives
\[
        \operatorname{tr}(\Phi^k)
        \in
        H^0\bigl(C,K_C^k((k-1)D)\bigr)
\]
whenever the residues of \(\Phi\) are nilpotent.  The following local
calculation shows that the pole order \(k-1\) cannot be improved in general.

\begin{proposition}
\label{prop:trace-sharp}
Let \(G=GL(n,\mathbb C)\), and let \(2\leq k\leq n\).  There are local logarithmic Higgs fields with nilpotent residue for which
\(
        \operatorname{tr}(\Phi^k)
\)
has a pole of order exactly \(k-1\).
\end{proposition}

\begin{proof}
It is enough to give a local logarithmic Higgs field with nilpotent residue for
which the coefficient of the pole of order \(k-1\) is nonzero.

Let
\[
        N=E_{12}+E_{23}+\cdots+E_{k-1,k}
        \in\mathfrak{gl}_k\subset\mathfrak{gl}_n
\]
be the nilpotent Jordan block of size \(k\). Choose a holomorphic matrix
\(A(z)\) with
\(
        A(0)=E_{k1}.
\)
Consider locally
\[
        \Phi=
        \left(
        \frac{N}{z}+A(z)
        \right)dz .
\]
The coefficient of \(z^{-(k-1)}dz^k\) in \(\operatorname{tr}(\Phi^k)\) is
obtained by taking exactly one factor \(A(0)\) and \(k-1\) factors \(N\). Hence
it is
\[
        \sum_{j=0}^{k-1}
        \operatorname{tr}
        \bigl(
        N^jA(0)N^{k-1-j}
        \bigr).
\]
By cyclicity of the trace, this equals
\[
        k\,\operatorname{tr}\bigl(N^{k-1}A(0)\bigr).
\]
Since
\[
        N^{k-1}=E_{1k},
        \qquad
        A(0)=E_{k1},
\]
we get
\[
        \operatorname{tr}\bigl(N^{k-1}A(0)\bigr)
        =
        \operatorname{tr}(E_{11})
        =
        1.
\]
Thus the coefficient is \(k\neq0\).  Therefore a pole of order \(k-1\) can occur.
\end{proof}

\begin{remark}
For \(k=2\), the nilpotent residue removes the double pole of
\(\operatorname{tr}(\Phi^2)\), but the simple pole may remain.  This is exactly
the local behaviour used later for strongly parabolic logarithmic
\(SL(2,\mathbb C)\)-Higgs fields.
\end{remark}

\subsection{The Pfaffian invariant}

Let
\(
        G=SO(2m,\mathbb C).
\)
The Pfaffian is the degree \(m\) invariant polynomial
\[
        \operatorname{Pf}:\mathfrak{so}(2m,\mathbb C)\longrightarrow\mathbb C.
\]
For a logarithmic \(SO(2m,\mathbb C)\)-Higgs field, the invariant
\(
        \operatorname{Pf}(\Phi)
\)
is therefore a meromorphic \(m\)-differential.  Theorem~\ref{thm:invpoly}
gives the pole estimate below. Under an additional rank condition on the
residues, the first Taylor term of the Pfaffian also vanishes, and the pole
order drops by one more.

\begin{corollary}
\label{cor:pfaffian-log}
Let \((P,\Phi)\) be a logarithmic \(SO(2m,\mathbb C)\)-Higgs bundle with
nilpotent residues. Then
\[
        \operatorname{Pf}(\Phi)
        \in
        H^0\bigl(C,K_C^m((m-1)D)\bigr).
\]
If, for every \(p_i\in D\),
\[
        \operatorname{rk}\operatorname{Res}_{p_i}(\Phi)\leq 2m-4,
\]
then
\[
        \operatorname{Pf}(\Phi)
        \in
        H^0\bigl(C,K_C^m((m-2)D)\bigr).
\]
\end{corollary}

\begin{proof}
The first statement follows from Theorem~\ref{thm:invpoly}, since
\(\operatorname{Pf}\) is homogeneous of degree \(m\). For the second assertion, put
\(
        N_i=\operatorname{Res}_{p_i}(\Phi).
\)
It is enough, by Lemma~\ref{lem:refined-pole}, to prove
\[
        d(\operatorname{Pf})_{N_i}=0.
\]
Choose a local orthogonal trivialization near \(p_i\), so that
\(N_i\) is represented by a skew-symmetric \(2m\times 2m\) matrix. For a skew-symmetric \(2m\times 2m\) matrix \(X\), the first partial
derivatives of \(\operatorname{Pf}(X)\) are, up to sign, Pfaffians of the
\((2m-2)\times(2m-2)\) skew-symmetric matrices obtained by deleting two
corresponding rows and columns. If
\[
        \operatorname{rk}(N_i)\leq 2m-4,
\]
then each such \((2m-2)\times(2m-2)\) matrix is singular. Its determinant is
zero, and hence its Pfaffian is zero. Therefore
\[
        d(\operatorname{Pf})_{N_i}=0.
\]
Then by Lemma~\ref{lem:refined-pole}, with \(p=\operatorname{Pf}\), \(d=m\), and
\(r_i=2\), we get the required bound.
\end{proof}

\begin{remark}
The quadratic invariant and the Pfaffian enter the logarithmic theory in
different degrees. The quadratic invariant gives
\[
        B(\Phi,\Phi)\in H^0(C,K_C^2(D)),
\]
which is the cotangent space to pointed Teichmuller space. The Pfaffian gives
\[
        \operatorname{Pf}(\Phi)\in H^0(C,K_C^m((m-1)D)),
\]
and, under the rank condition above,
\[
        \operatorname{Pf}(\Phi)\in H^0(C,K_C^m((m-2)D)).
\]
Thus, for \(m>2\), the Pfaffian belongs to the higher Hitchin Hamiltonians
rather than to the ordinary Teichmuller cotangent direction.  In the exceptional
case \(m=2\), the Pfaffian also has degree two.
\end{remark}

\section{The logarithmic quadratic one-form}
\label{sec:logoneform}

Let
\[
        \pi:\mathcal M^{\mathrm{nil}}_{\mathrm{Dol}}(G)
        \longrightarrow
        \mathcal T_{g,r}
\]
be the smooth stable family fixed in Section~\ref{sec:prelim}. Thus the fibre
over \((C,D)\) consists of logarithmic \(G\)-Higgs bundles
\(
        (P,\Phi),
        \ \text{where} \
        \Phi\in H^0(C,\operatorname{ad}(P)\otimes K_C(D)),
\)
with nilpotent residues. By Corollary~\ref{cor:quadratic},
\[
        B(\Phi,\Phi)\in H^0(C,K_C^2(D)).
\]
On the other hand,
\[
        T^*_{(C,D)}\mathcal T_{g,r}
        \cong
        H^0(C,K_C^2(D)).
\]
Hence the quadratic differential
\(
        -\frac12B(\Phi,\Phi)
\)
is a cotangent vector to pointed Teichmuller space at \((C,D)\).  We denote the
corresponding one-form on the family by
\(
        \phi_{G,\log}.
\)
Thus, for
\[
        Y\in T_{(C,D)}\mathcal T_{g,r}
        \cong H^1(C,T_C(-D)),
\]
we set
\[
        \phi_{G,\log}(Y)
        =
        -\frac12
        \left\langle B(\Phi,\Phi),Y\right\rangle .
\]

\begin{theorem}
\label{thm:log-oneform}
The one-form
\(
        \phi_{G,\log}
\)
is a holomorphic section of
\(
        \pi^*\Omega^1_{\mathcal T_{g,r}}
\)
over
\(
        \mathcal M^{\mathrm{nil}}_{\mathrm{Dol}}(G).
\)
If \(Y\) is represented by a Dolbeault form
\[
        \nu\in A^{0,1}(C,T_C(-D))
\]
which vanishes near \(D\), then
\[
        \phi_{G,\log}(Y)
        =
        -\frac12\int_C B(\Phi,\Phi)\nu .
\]
\end{theorem}

\begin{proof}
The first assertion follows from the polynomial dependence of
\(
        B(\Phi,\Phi)
\)
on the Higgs field, together with the Serre-duality identification
\[
        H^0(C,K_C^2(D))
        \cong
        T^*_{(C,D)}\mathcal T_{g,r}.
\]
Thus \(\phi_{G,\log}\) varies holomorphically along the smooth relative moduli
space. If \(\nu\) vanishes near \(D\), then the product \(B(\Phi,\Phi)\nu\) is a smooth
\((1,1)\)-form on \(C\).  In this case the Serre pairing is represented by the integral
\[
        \left\langle B(\Phi,\Phi),Y\right\rangle
        =
        \int_C B(\Phi,\Phi)\nu,
\]
which gives the required formula.
\end{proof}

\subsection{The local residue formula}

Let
\(
        Y\in H^1(C,T_C(-D))
\)
be represented by
\(
        (\mu,\xi_1,\ldots,\xi_r),
\)
where \(\mu\) is a Beltrami differential on \(C^\circ\), and \(\xi_i\) is a
local holomorphic vector field near \(p_i\).

\begin{proposition}
\label{prop:local-serre-formula}
Let
\(
        q=B(\Phi,\Phi),
\)
then
\[
        \phi_{G,\log}(Y)
        =
        -\frac12
        \left[
        \int_{C^\circ} B(\Phi,\Phi)\mu
        +
        \sum_{i=1}^{r}
        \operatorname{Res}_{p_i}
        \bigl(B(\Phi,\Phi)\xi_i\bigr)
        \right].
\]
\end{proposition}

\begin{proof}
By Corollary~\ref{cor:quadratic}, we have
\[
        B(\Phi,\Phi)\in H^0(C,K_C^2(D)).
\]
From the definition of \(\phi_{G,\log}\), we get
\[
        \phi_{G,\log}(Y)
        =
        -\frac12
        \left\langle B(\Phi,\Phi),Y\right\rangle .
\]
Therefore, by applying the local form of the Serre pairing from
\eqref{eq:serre-local}, we get
\[
        \left\langle B(\Phi,\Phi),Y\right\rangle
        =
        \int_{C^\circ} B(\Phi,\Phi)\mu
        +
        \sum_{i=1}^{r}
        \operatorname{Res}_{p_i}
        \bigl(B(\Phi,\Phi)\xi_i\bigr).
\]
This proves the formula.
\end{proof}

We now give the following results which will be used later.

\begin{proposition}
\label{prop:residue-well-defined}
Let
\(
        q\in H^0(C,K_C^2(D)).
\)
Then
\(
        \operatorname{Res}_{p_i}(q\,\xi_i)
\)
depends only on the tangent vector
\(
        \xi_i(p_i)\in T_{p_i}C.
\)
\end{proposition}

\begin{proof}
Let \(\xi_i'\) be another local holomorphic vector field near \(p_i\) such that
\(
        \xi_i'(p_i)=\xi_i(p_i).
\)
Then
\(
        \eta_i=\xi_i-\xi_i'
\)
vanishes at \(p_i\). In a coordinate \(z\) centred at \(p_i\), write
\[
        \eta_i=z\,g(z)\frac{\partial}{\partial z}
\]
with \(g\) holomorphic. Since \(q\) has at most a simple pole, locally
\[
        q=
        \left(
        \frac{a_i}{z}+h_i(z)
        \right)dz^2
\]
with \(h_i\) holomorphic. Hence
\[
        q\,\eta_i
        =
        \left(
        a_i g(z)+z h_i(z)g(z)
        \right)dz,
\]
which is holomorphic at \(p_i\). Therefore
\[
        \operatorname{Res}_{p_i}(q\,\eta_i)=0.
\]
The residue is unchanged when \(\xi_i\) is replaced by \(\xi_i'\).
\end{proof}

The integral term in Proposition~\ref{prop:local-serre-formula} is also
well-defined for bounded Beltrami representatives.

\begin{proposition}
\label{prop:local-convergence}
Let
\(
        q\in H^0(C,K_C^2(D)).
\)
If \(\mu\) is bounded near \(p_i\), then
\(
        \int q\,\mu
\)
is locally absolutely convergent near \(p_i\).
\end{proposition}

\begin{proof}
Choose a coordinate \(z\) centred at \(p_i\). Locally,
\[
        q=
        \left(
        \frac{a}{z}+h(z)
        \right)dz^2,
        \qquad
        \mu=
        b(z)d\bar z\otimes\frac{\partial}{\partial z},
\]
where \(h\) is holomorphic and \(b\) is bounded. Then
\[
        |q\,\mu|
        \leq
        C|z|^{-1}\,dx\,dy
\]
near \(z=0\). In polar coordinates this is bounded by
\(
        C r^{-1}r\,dr\,d\theta,
\)
which is integrable near \(r=0\).  This proves the local convergence.
\end{proof}

\subsection{Circle action}

The circle action
\(
        (P,\Phi)\longmapsto (P,e^{i\theta}\Phi)
\)
preserves the nilpotent-residue locus. Indeed,
\[
        \operatorname{Res}_{p_i}(e^{i\theta}\Phi)
        =
        e^{i\theta}\operatorname{Res}_{p_i}(\Phi),
\]
and scalar multiples of nilpotent elements are nilpotent.

\begin{proposition}
\label{prop:circle-weight-log}
Under this circle action,
\(
        \phi_{G,\log}
        \longmapsto
        e^{2i\theta}\phi_{G,\log}.
\)
\end{proposition}

\begin{proof}
The invariant bilinear form \(B\) gives
\[
        B(e^{i\theta}\Phi,e^{i\theta}\Phi)
        =
        e^{2i\theta}B(\Phi,\Phi).
\]
Since \(\phi_{G,\log}\) is obtained from
\[
        -\frac12B(\Phi,\Phi)
\]
by the Serre-duality pairing with \(H^1(C,T_C(-D))\), the one-form has weight
two under the circle action.
\end{proof}

\subsection{The logarithmic Hitchin map}

Let
\(
        p_1,\ldots,p_\ell
\)
be homogeneous generators of
\(
        \mathbb C[\mathfrak g]^G,
\)
with degrees
\(
        d_1,\ldots,d_\ell.
\)
For logarithmic Higgs fields with arbitrary residues, the invariant
\(p_j(\Phi)\) has, a priori, pole order at most \(d_j\) along \(D\). Thus the
usual logarithmic Hitchin map takes values in
\[
        \bigoplus_{j=1}^{\ell}
        H^0(C,K_C^{d_j}(d_jD)).
\]
By Corollary~\ref{cor:nilpotent-hitchin-map}, on the nilpotent-residue locus
this map factors through the smaller space
\[
        \bigoplus_{j=1}^{\ell}
        H^0(C,K_C^{d_j}((d_j-1)D)).
\]
The degree-two invariant is the part relevant to pointed Teichmuller space.
Indeed, for
\(
        p_B(X)=\frac12B(X,X),
\)
we have
\[
        B(\Phi,\Phi)\in H^0(C,K_C^2(D)).
\]
Therefore, together with
\[
        H^0(C,K_C^2(D))
        \cong
        T^*_{(C,D)}\mathcal T_{g,r},
\]
this gives the logarithmic quadratic one-form
\(
        \phi_{G,\log}.
\)

\section{Logarithmic energy variation}
\label{sec:energy}

We prove the energy-variation formula for tame nilpotent harmonic bundles on
the punctured curve
\(
        C^\circ=C\setminus D.
\)
The algebraic construction of \(\phi_{G,\log}\) was independent of harmonic
metric estimates. The energy variation uses the tame asymptotic behaviour near
the punctures, together with the decay assumption stated below.

A tame harmonic bundle on \(C^\circ\) is a flat \(G\)-bundle together
with a harmonic metric whose associated Higgs field has logarithmic growth at
the punctures. In this paper we use only the nilpotent-residue case. Thus we consider a logarithmic Higgs bundle
\(
        (P,\Phi),
        \ \text{with}\
        \Phi\in H^0(C,\operatorname{ad}(P)\otimes K_C(D)),
\)
together with a harmonic metric \(h\) on \(C^\circ\), such that every residue of
\(\Phi\) is nilpotent. 

Let \(d_A\) be the Chern connection determined by the harmonic metric
\(h\), and let \(\tau\) denote the adjoint operation induced by \(h\). If we take
\(
        \varphi=\Phi+\tau(\Phi),
\)
then the associated flat connection is
\(
        \nabla=d_A+\varphi.
\)
Near a puncture, in a coordinate \(z\), the Higgs field has the form
\[
        \Phi=
        \left(
        \frac{N}{z}+A(z)
        \right)dz,
\]
where \(N\) is nilpotent and \(A(z)\) is holomorphic. With respect to the
complete cusp metric on \(C^\circ\),
\[
        |\varphi|_{h,\mathrm{cusp}}=O(1)
\]
near each puncture (see \cite{Simpson1990,MochizukiNilpotent}). 

We consider variations with fixed local monodromy conjugacy classes. The
infinitesimal variation of the harmonic metric is then expressed by a complex
gauge term. The required decay of this term is stated in
Assumption~\ref{ass:cusp-slice}.

\subsection{The cusp scale and the gauge term}

Near a point of \(D\), write
\[
        z=re^{i\theta},
        \qquad
        x=(-\log |z|^2)^{-1}.
\]
Then \(x\to0\) as \(z\to0\).  The Poincare metric near the puncture is
quasi-isometric to
\[
        \frac{dx^2}{x^2}+x^2d\theta^2.
\]
The vector fields of bounded length for this metric are generated locally by
\(
        x\partial_x,
        \ \text{and} \
        x^{-1}\partial_\theta.
\)
We denote by
\(
        C_{\mathrm{cusp}}^{k,\alpha}
\)
the Hölder space defined using these vector fields and the cusp distance.  For
\(\delta>0\), the weighted space
\(
        x^\delta C_{\mathrm{cusp}}^{k,\alpha}
\)
consists of sections \(u\) such that
\(
        x^{-\delta}u\in C_{\mathrm{cusp}}^{k,\alpha}.
\)

\begin{assumption}
\label{ass:cusp-slice}
For the variation under consideration, the local monodromy conjugacy classes
are fixed. The infinitesimal complex gauge term \(\psi\) satisfies, near each
puncture,
\(
        \psi\in x^\delta C_{\mathrm{cusp}}^{2,\alpha}
\)
for some
\(
        \delta>0,
        \
        0<\alpha<1.
\)
\end{assumption}

This assumption is used only in the integration-by-parts argument below, where it makes the boundary terms tend to zero.

\subsection{Finite energy}

Define
\[
        f_{\mathrm{par}}
        =
        -\frac12\int_{C^\circ}B(\varphi,*\varphi),
\]
provided the integral is absolutely convergent.

\begin{proposition}
\label{prop:finite-energy}
For a tame nilpotent harmonic bundle as above, the integral defining
\(f_{\mathrm{par}}\) is absolutely convergent.
\end{proposition}

\begin{proof}
It is enough to prove convergence near each point of \(D\). By the tame
nilpotent estimate,
\[
        |\varphi|_{h,\mathrm{cusp}}=O(1)
\]
near a puncture. Since \(B\) is fixed, there is a constant \(C>0\) such that
\[
        |B(\varphi,*\varphi)|
        \leq
        C|\varphi|_{h,\mathrm{cusp}}^2\,d\mathrm{vol}_{\mathrm{cusp}}.
\]
Thus the local contribution to the energy is bounded by a constant multiple of
the cusp area form.

In a punctured coordinate disc, the cusp area form is comparable to
\[
        \frac{dr\,d\theta}{r(\log r)^2}.
\]
Hence the local integral is bounded by
\[
        C\int_0^\epsilon \frac{dr}{r(\log r)^2},
\]
which is finite. Since \(D\) is finite, the integral over \(C^\circ\) is
absolutely convergent.
\end{proof}

\subsection{Boundary terms from the gauge variation}

Choose pairwise disjoint coordinate discs around the points of \(D\), and put
\[
        C_\varepsilon
        =
        C\setminus\bigcup_{i=1}^{r}\{|z_i|<\varepsilon\}.
\]
The first-variation formula is first applied on the compact surface
\(C_\varepsilon\).  The integration-by-parts terms coming from the complex gauge
variation are supported on \(\partial C_\varepsilon\).  The following estimate
is the only place where Assumption~\ref{ass:cusp-slice} is used.

\begin{lemma}
\label{lem:gauge-boundary}
Under Assumption~\ref{ass:cusp-slice}, the boundary terms arising from the
complex gauge variation over \(C_\varepsilon\) tend to zero as
\(\varepsilon\to0\).
\end{lemma}

\begin{proof}
It is enough to estimate the contribution near one puncture. Let
\(
        x=(-\log |z|^2)^{-1}.
\)
On the circle \(|z|=\varepsilon\), we have \(x=x_\varepsilon\), where
\(x_\varepsilon\to0\) as \(\varepsilon\to0\). The cusp length of this circle is
\(
        O(x_\varepsilon).
\)
The tame nilpotent estimate gives
\[
        |\varphi|_{h,\mathrm{cusp}}=O(1).
\]
By Assumption~\ref{ass:cusp-slice},
\[
        |\psi|_h=O(x_\varepsilon^\delta)
\]
on \(|z|=\varepsilon\), for some \(\delta>0\). The boundary terms produced by the gauge part of the variation are finite sums
of integrals of the form
\[
        \int_{|z|=\varepsilon}
        B(\psi,\iota_n\varphi)\,ds_{\mathrm{cusp}},
        \qquad
        \int_{|z|=\varepsilon}
        B(\psi,\iota_n *\varphi)\,ds_{\mathrm{cusp}},
\]
where \(n\) is the cusp unit normal. On \(|z|=\varepsilon\), the contraction with the cusp unit normal satisfies
\[
        |\iota_n\varphi|_{h}\leq |\varphi|_{h,\mathrm{cusp}},
        \qquad
        |\iota_n *\varphi|_{h}\leq |\varphi|_{h,\mathrm{cusp}}.
\]
Hence, for a boundary term of the first type, we have
\[
\begin{aligned}
\left|
        \int_{|z|=\varepsilon}
        B(\psi,\iota_n\varphi)\,ds_{\mathrm{cusp}}
\right|
&\leq
C
\sup_{|z|=\varepsilon}|\psi|_h\,
\sup_{|z|=\varepsilon}|\varphi|_{h,\mathrm{cusp}}\,
\operatorname{length}_{\mathrm{cusp}}(|z|=\varepsilon)  \\
&\leq
C\,x_\varepsilon^\delta\cdot 1\cdot x_\varepsilon
=
O(x_\varepsilon^{1+\delta}).
\end{aligned}
\]
The same estimate applies to the term with \(\iota_n *\varphi\). Since
\(\delta>0\), these boundary integrals tend to zero as \(\varepsilon\to0\). Summing over the finitely many punctures proves the lemma.
\end{proof}

\subsection{The variation formula}

Let
\[
        Y\in T_{(C,D)}\mathcal T_{g,r}
        \cong
        H^1(C,T_C(-D)).
\]
By Lemma~\ref{lem:vanishing-representative}, choose a representative
\[
        \nu\in A^{0,1}(C,T_C(-D))
\]
which vanishes on a neighbourhood of \(D\).

\begin{lemma}
\label{lem:first-variation-limit}
Let
\(
        Y\in H^1(C,T_C(-D))
\)
be represented by a Dolbeault form
\(
        \nu\in A^{0,1}(C,T_C(-D))
\)
which vanishes on a neighbourhood of \(D\).  Suppose that the Betti point and
the local monodromy conjugacy classes are fixed, and that
Assumption~\ref{ass:cusp-slice} holds.  Then the first variations of the
truncated energies satisfy
\[
        \lim_{\varepsilon\to0}Y(f_{\mathrm{par},\varepsilon})
        =
        Y(f_{\mathrm{par}}).
\]
\end{lemma}

\begin{proof}
Since \(\nu\) is zero on a neighbourhood of \(D\), the complex-structure part of
the first variation has no contribution from the cusp region
\(
        C^\circ\setminus C_\varepsilon
\)
for all sufficiently small \(\varepsilon\).

It remains to consider the terms coming from the complex gauge variation.  The
tame nilpotent estimate gives
\(
        |\varphi|_{h,\mathrm{cusp}}=O(1),
\)
and Assumption~\ref{ass:cusp-slice} gives
\(
        |\psi|_h=O(x^\delta)
\)
for some \(\delta>0\).  The same estimate as in
Lemma~\ref{lem:gauge-boundary} shows that the contribution of the cusp region
to the first variation tends to zero as \(\varepsilon\to0\).  Hence the first
variation of the truncated energies converges to the first variation of the
finite-energy functional.
\end{proof}

\begin{theorem}
\label{thm:log-energy}
Let \((P,\Phi,h)\) be a tame nilpotent harmonic bundle on \(C^\circ\). Consider
a first-order deformation in the direction
\(
        Y\in T_{(C,D)}\mathcal T_{g,r}
\)
with fixed Betti representation and fixed local monodromy conjugacy classes.
Assume that Assumption~\ref{ass:cusp-slice} holds for the corresponding
infinitesimal complex gauge term. Then
\[
        Y(f_{\mathrm{par}})
        =
        -\frac12\operatorname{Re}
        \int_C B(\Phi,\Phi)\nu.
\]
Equivalently,
\[
        Y(f_{\mathrm{par}})
        =
        \operatorname{Re}\phi_{G,\log}(Y).
\]
\end{theorem}

\begin{proof}
Choose disjoint coordinate discs around the points of \(D\), and put
\[
        C_\varepsilon
        =
        C\setminus\bigcup_i\{|z_i|<\varepsilon\}.
\]
For sufficiently small \(\varepsilon\), the form \(\nu\) vanishes in a
neighbourhood of \(\partial C_\varepsilon\).  Define the truncated energy by
\[
        f_{\mathrm{par},\varepsilon}
        =
        -\frac12
        \int_{C_\varepsilon}B(\varphi,*\varphi).
\]
On \(C_\varepsilon\), all fields are smooth and the surface is compact with
boundary. Since the Betti representation is fixed, the infinitesimal variation
of the flat connection is a complex gauge variation, as in
Proposition~\ref{prop:compact-variation}. Then by the first-variation formula
on \(C_\varepsilon\), we get
\[
        Y(f_{\mathrm{par},\varepsilon})
        =
        -\frac12\operatorname{Re}
        \int_{C_\varepsilon}B(\Phi,\Phi)\nu
        +
        \mathcal B_\varepsilon,
\]
where \(\mathcal B_\varepsilon\) is the sum of the boundary terms coming from
the gauge variation. There is no boundary contribution from the variation of
complex structure, because \(\nu\) vanishes near \(\partial C_\varepsilon\).

By Lemma~\ref{lem:gauge-boundary},
\[
        \mathcal B_\varepsilon\longrightarrow0
        \qquad
        \text{as }\varepsilon\to0.
\]
Since \(\nu\) vanishes near \(D\), the integral
\[
        \int_{C_\varepsilon}B(\Phi,\Phi)\nu
\]
is independent of \(\varepsilon\) for all sufficiently small \(\varepsilon\), and
is equal to
\[
        \int_C B(\Phi,\Phi)\nu.
\]
By Lemma~\ref{lem:first-variation-limit},
\[
        \lim_{\varepsilon\to0}Y(f_{\mathrm{par},\varepsilon})
        =
        Y(f_{\mathrm{par}}).
\]
Passing to the limit in the first-variation formula therefore gives
\[
        Y(f_{\mathrm{par}})
        =
        -\frac12\operatorname{Re}
        \int_C B(\Phi,\Phi)\nu.
\]
The equality
\[
        Y(f_{\mathrm{par}})
        =
        \operatorname{Re}\phi_{G,\log}(Y)
\]
now follows from Theorem~\ref{thm:log-oneform}.
\end{proof}

Combining this with the local expression of
\(\phi_{G,\log}\) gives the residue form of the variation formula.

\begin{corollary}
\label{cor:log-energy-residue}
Let
\(
        Y\in H^1(C,T_C(-D))
\)
be represented by
\(
        (\mu,\xi_1,\ldots,\xi_r),
\)
where \(\mu\) is a Beltrami differential on \(C^\circ\), and \(\xi_i\) is a
local holomorphic vector field near \(p_i\). Under the hypotheses of
Theorem~\ref{thm:log-energy},
\[
        Y(f_{\mathrm{par}})
        =
        -\frac12\operatorname{Re}
        \left[
        \int_{C^\circ}B(\Phi,\Phi)\mu
        +
        \sum_{i=1}^{r}
        \operatorname{Res}_{p_i}
        \bigl(B(\Phi,\Phi)\xi_i\bigr)
        \right].
\]
\end{corollary}

\begin{proof}
By Theorem~\ref{thm:log-energy},
\[
        Y(f_{\mathrm{par}})
        =
        \operatorname{Re}\phi_{G,\log}(Y).
\]
The formula follows from Proposition~\ref{prop:local-serre-formula}.
\end{proof}

\begin{remark}
The pole-order theorem and the construction of \(\phi_{G,\log}\) do not use
Assumption~\ref{ass:cusp-slice}. The assumption is used only in the proof of
the energy-variation formula, to make the gauge boundary terms vanish.
\end{remark}

\section*{Acknowledgment}

The author is supported by the INSPIRE faculty fellowship (Ref No.: IFA22-MA 186) funded by the DST, Govt. of India.

\end{document}